\documentclass[a4paper,11pt]{amsart}

\usepackage{amsmath}
\usepackage{amssymb}
\usepackage{amsthm}
\usepackage[all,cmtip]{xy}
\usepackage{textcomp}
\usepackage{slashed}
\usepackage{enumerate}
\usepackage{lipsum}
\usepackage{caption}

\newtheorem{thm}{Theorem}[section]
\newtheorem{lemma}[thm]{Lemma}
\newtheorem{df}{Definition}[section]
\newtheorem{remark}{Remark}[section]

\newtheorem{prop}[thm]{Proposition}
\numberwithin{equation}{section}

\newcommand{\lag}{\mathfrak g}
\newcommand{\lap}{\mathfrak p}

\newcommand{\lau}{\mathfrak u}

\newcommand{\laso}{\mathfrak{so}}

\newcommand{\lah}{\mathfrak h}

\newcommand{\lagl}{\mathfrak{gl}}

\newcommand{\mW}{\mathbb W}
\newcommand{\mV}{\mathbb V}
\newcommand{\mE}{\mathbb E}
\newcommand{\mF}{\mathbb F}

\newcommand{\gO}{\mathrm{O}}
\newcommand{\gP}{\mathrm P}
\newcommand{\gK}{\mathrm K}

\newcommand{\gH}{\mathrm H}
\newcommand{\gU}{\mathrm{U}}

\newcommand{\rgr}{G_2^+(\mathbb{R}^{n+2})}
\newcommand{\pgr}{G_2^o(\mathbb{R}^{2,n+2})}
\newcommand{\sv}{V_2(\mathbb{R}^{n+2})}
\newcommand{\CSO}{\mathrm{CSO}}
\newcommand{\Ad}{\mathrm{Ad}}

\newcommand{\SO}{\mathrm{SO}}

\newcommand{\SL}{\mathrm{SL}}
\newcommand{\Aut}{\mathrm{Aut}}

\newcommand{\GL}{\mathrm{GL}}

\newcommand{\CSp}{\mathrm{CSp}}
\newcommand{\Spin}{\mathrm{Spin}}

\newcommand{\G}{\mathrm G}

\newcommand{\R}{\mathbb R}

\newcommand{\Sp}{\mathbb S}

\newcommand{\Z}{\mathbb Z}

\newcommand{\cE}{\mathcal E}

\newcommand{\cP}{\mathcal P}
\newcommand{\cS}{\mathcal S}

\newcommand{\cG}{\mathcal G}

\newcommand{\cp}{p^\sharp}
\newcommand{\ccG}{\cG^\sharp}

\newcommand{\cL}{\mathcal L}

\newcommand{\cgK}{\mathrm K^\sharp}
\newcommand{\cgH}{\mathrm H^\sharp}
\newcommand{\tcgK}{\tilde{\mathrm K}^\sharp}
\newcommand{\tcgH}{\tilde{\mathrm H}^\sharp}
\newcommand{\tccG}{\tilde{\mathcal G}^\sharp}
\newcommand{\ra}{\rightarrow}
\newcommand{\cm}{M^\sharp}
\newcommand{\cx}{x^\sharp}
\newcommand{\tcp}{\tilde{p}^\sharp}

\newcommand{\Id}{\operatorname{Id}}
\newcommand{\End}{\operatorname{End}}
\newcommand{\im}{\operatorname{im}}

\title{Elliptic complex on the Grassmannian of oriented 2-planes}
\author{Tom\'a\v s Sala\v c }
\thanks{
The author gratefully acknowledges the support of FWF grant P23244–N13 and 17-01171S of the Grant Agency of the Czech Republic.
\\}

\begin{document}

\begin{abstract} 
The Grassmannian $\rgr$ of oriented 2-planes in $\R^{n+2}$ where $n\ge3$ carries a homogeneous parabolic conformally symplectic structure of Grassmannian type. The main result of this article is that on $\rgr$ lives an elliptic complex of invariant differential operators of length 3  which starts with the 2-Dirac operator and that the index of the complex is zero. 
\end{abstract}
\maketitle

\section{Introduction}
A parabolic conformally symplectic structure (or PCS-structure for short) of Grassmannian type on a manifold $M$ of dimension $2n\ge6$ is a Grassmannian structure with auxiliary (oriented) vector bundles $E$ and $F$ of rank $2$ and $n$, respectively, together with a conformally symplectic structure which is Hermitian in the Grassmannian sense, see Section \ref{section PACS structures of GT}. In particular, there is an isomorphism  $TM\cong E^\ast\otimes F$ and  a conformal class of (positive definite) bundle metrics on $F$.
This geometric structure is equivalent to $\G_0:=\GL^+(2,\R)\times\SO(n)$-structure  which satisfies some integrability condition. It turns out that there is a canonical compatible linear connection on $TM$ which can be equivalently viewed as a principal connection on the $\G_0$-structure. If the $\G_0$-structure lifts to a $\tilde\G_0:=\GL^+(2,\R)\times\Spin(n)$-structure $\tilde\cG_0\ra M$, then the principal connection lifts (in a unique way) to a principal connection on the $\tilde\G_0$-structure and thus, it induces a covariant derivative $\tilde\nabla$ on any associated vector bundle. 

Let  $\Sp$ be a complex spinor representation of $\Spin(n)$. We extend this representation to $\tilde\G_0$, see Section  \ref{section Dirac complex} for details, and put $\cS:=\tilde\cG_0\times_{\tilde\G_0}\Sp$. As $T^\ast M$ is associated to a ${\tilde\G_0}$-representation with underlying vector space $\R^{2}\otimes\R^{n\ast}$, see Section \ref{section PACS structures of GT}, there is  a vector bundle map $\gamma:T^\ast M\otimes\cS\ra E\otimes\cS$ that is induced by a $\tilde\G_0$-equivariant map $\R^{2}\otimes\R^{n\ast}\otimes\Sp\ra\R^{2}\otimes\Sp$. It follows that there is a linear differential operator of first order, called the 2-\textit{Dirac operator}, which is  invariantly  defined by
\begin{equation}\label{2-Dirac operator}
 \slashed{D}:\Gamma(\cS)\xrightarrow{\tilde\nabla}\Gamma(T^\ast M\otimes\cS)\ra\Gamma(E\otimes\cS)
\end{equation}
where the second map is induced by $\gamma$.
If $x\in M$, $\{e_1,e_2\}$ is a basis of $E^\ast_x$, $\{\varepsilon_1,\dots,\varepsilon_n\}$ is an orthonormal basis\footnote{With respect to some metric in the conformal class.} of $F_x$ and $\psi\in\Gamma(\cS)$, then 
\begin{equation}\label{2-Dirac operator locally}
\slashed{D}\psi(x)=\sum_{\alpha=1}^n(\varepsilon_\alpha.\tilde\nabla_{e_1\otimes\varepsilon_\alpha}\psi,\varepsilon_\alpha.\tilde\nabla_{e_2\otimes\varepsilon_\alpha}\psi)(x).
\end{equation}
Here we use the Grassmannian structure so that we can view $e_i\otimes\varepsilon_\alpha\in T_xM$ and $\varepsilon_\alpha.\in\End(\Sp)$ denotes the usual Clifford multiplication on $\Sp$. 

If $M=\R^{2n}$ with its flat PCS-structure of Grassmannian type, then $\slashed{D}$ is the 2-Dirac operator studied in \cite{CSSS}.
The $2$-Dirac operator can be viewed as a generalization of the Dirac operator in Riemannian geometry. On the other hand, $\slashed{D}$ is overdetermined and it turns out that it is the first operator in a sequence of natural differential operators. This sequence is, as explained in the series  \cite{CSI}, \cite{CSII} and \cite{CSIII} together with a preliminary article \cite{CSa}, obtained by  descending the sequence of natural differential operators studied in \cite{S} from a geometric structure in one dimension higher. If $M=\R^{2n}$, then the descended sequence is locally exact and thus, it forms (see \cite{TSI}) a resolution of the 2-Dirac operator. 

In this paper is considered  a similar sequence of invariant operators which lives on the Grassmannian $\rgr$ of oriented 2-planes in $\R^{n+2}$. The Grassmannian  carries a homogeneous  PCS-structure of Grassmannian type with structure group $\G_0$ which \textit{does not} lift to  $\tilde\G_0$-structure. Nevertheless, it lifts to a  $\G_0^c:=\GL^+(2,\R)\times\Spin^c(n)$-structure and since the complex spinor representation of $\Spin(n)$ extends to $\Spin^c(n)$, there is the spinor bundle associated to the $\G_0^c$-structure. Fixing a principal connection  which lifts the canonical connection,  we can define  the 2-Dirac operator over $\rgr$  as in (\ref{2-Dirac operator}).

We will show (see Theorem \ref{thm descended complex}) that the 2-Dirac operator over $\rgr$ is the first operator in an elliptic complex of linear invariant operators and we will prove (see Theorem \ref{thm ellipticity and index}) that  its index is zero. In order to construct the complex, we  will adapt (see also Remark \ref{remark not pcs quotient}) the descending scheme introduced in the series \cite{CSI,CSII,CSIII} and we will show 
that the elliptic complex  descends  from the 2-Dirac complex introduced in \cite{S}. 

\bigskip
The author is grateful to Andreas \v Cap and Micheal C. Crabb for valuable comments and suggestions. The author wishes to thank the unknown referees for their valuable suggestions which considerably improved the current manuscript.

\bigskip

\begin{center}
 \textbf{Notation}
\end{center}
\begin{tabular}{rl}
$1_k$& identity $k\times k$ matrix\\
$M_{n\times k}(\mathbb F)$& matrices of size $n\times k$ with coefficients in field $\mathbb F$\\
$M_{k}(\mathbb F)$&$:= M_{k\times k}(\mathbb F)$\\
$A_{k}(\mathbb F)$& skew-symmetric $k\times k$ matrices with coefficients in field $\mathbb F$\\
$\ker(\psi), \im(\psi)$& kernel and image of map $\psi$, respectively
\end{tabular}

\section{Preliminaries}\label{section preliminaries}
In Section \ref{section preliminaries}  we will recall some well known material on the group $\Spin^c(n)$, spin and spin$^c$ structures.
This material can be found for example in \cite{Mo}. 
 
\subsection{Spin, spin$^c$ groups and spinors}\label{section spin group}
Let $\rho_n:\Spin(n)\ra\SO(n), n\ge2$ be the standard 2:1 covering with $\ker(\rho_n)=\{\pm1\}$. If $n\ge3$, then $\Spin(n)$ is a universal cover of $\SO(n)$ and  $\ker(\rho_n)$ is the center of $\Spin(n)$.  If $n=2$, then $\Spin(2)\cong\SO(2)\cong \gU(1)=\{e^{it}|\ t\in\R\},\ \rho_2(e^{it})=e^{2it}$  and $\laso(2)\cong\lau(1)=\{it:\ t\in\R\}$. We will for brevity write $-a:=(-1).a=a.(-1),\ a\in\Spin(n)$.

The group $\Spin^c(n)$ is the quotient of $\gU(1)\times\Spin(n)$ by the normal subgroup $\Z_2=\{\pm(1,1)\}$ and so there is a short exact sequence 
\begin{equation}
0\ra\Z_2\ra\gU(1)\times\Spin(n)\xrightarrow{\pi}\Spin^c(n)\ra0
\end{equation}
where we for a moment denote by $\pi$ the canonical projection. As $\Z_2$ is a discrete subgroup,  the Lie algebra of $\Spin^c(n)$ is $\lau(1)\oplus\laso(n)$. We for brevity put $\langle e^{it},a\rangle:=\pi(e^{it},a)$ so that $\langle e^{it},a\rangle=\langle e^{is},b\rangle$ if and only if $a=b,\ e^{it}=e^{is}$ or $a=-b,\ e^{it}=-e^{is}$.

The maps $\Spin(n)\ra\Spin^c(n),\ a\mapsto\langle 1,a\rangle$ and $\gU(1)\ra\Spin^c(n),\ e^{it}\mapsto\langle e^{it},1\rangle$ are clearly injective  homomorphisms of Lie group and so we may view  $\Spin(n)$ and $\gU(1)$ as subgroups of $\Spin^c(n)$.  If $n\ge3$, then $\gU(1)$ is the center of $\Spin^c(n)$. We obtain short exact sequences:
\begin{equation}
0\ra\gU(1)\ra\Spin^c(n)\xrightarrow{\rho_n^c}\SO(n)\ra0\label{first ses}
\end{equation}
and
\begin{equation}
0\ra\Spin(n)\ra\Spin^c(n)\xrightarrow{\varsigma_n}\gU(1)\ra0\label{second ses}
\end{equation}
where $\rho_n^c(\langle e^{it},a\rangle)=\rho_n(a)$ and $\varsigma_n(\langle e^{it},a\rangle)=\rho_2(e^{it})$.

\begin{lemma}\label{lemma inclusion of spinc inside spin}
There is a short exact sequence of Lie groups
\begin{equation}\label{third ses}
 0\ra\Z_2\ra\Spin^c(n)\xrightarrow{\varsigma_n\times\rho_n^c}\SO(2)\times\SO(n)\ra0
\end{equation}
and a commutative diagram 
\begin{equation}
\xymatrix{\Spin^c(n)\ar[d]^{\varsigma_n\times\rho_n^c}\ar[r]^{\iota}&\Spin(n+2)\ar[d]^{\rho_{n+2}}\\
\SO(2)\times\SO(n)\ar[r]^{\triangle}&\SO(n+2)}
\end{equation}
where $\iota$ is an embedding and $\triangle$ is the standard block diagonal embedding. 
\end{lemma}
\begin{proof}
The first claim is clear. To prove the second claim, we  need to show that $\rho_{n+2}^{-1}(\triangle(\SO(2)\times\SO(n))\cong\Spin^c(n)$ and this is straightforward.
\end{proof}

Let $(\mW,\gamma)$ be a  complex representation of $\Spin(n)$. If $\gamma(-1)=-\Id_\mW$, then 
\begin{eqnarray}\label{spinc representation on spinors}
\gamma^c:\Spin^c(n)\ra\GL(\mW), \
\gamma^c(\langle e^{it},a\rangle)(v)=e^{it}\gamma(a)v,
\end{eqnarray}
where $a\in\Spin(n),\ t\in\R$ and $v\in\mW$, is a complex representation of $\Spin^c(n)$. In particular, this applies for  $\mW=\Sp$.

\subsection{Spin and spin$^c$-structure}\label{section spinc structure}
Let $q:\cE\ra M$ be a $\gU(1)$-principal bundle (or simply a circle bundle), $\partial_t$ be the fundamental vector field corresponding to $i\in\lau(1)$, $\Omega^p(\cE)^{\gU(1)}$ be the space of  $\gU(1)$-invariant $p$-forms on $\cE$  and $i_\xi$ be the insertion operator associated to a vector field $\xi\in\mathfrak X(\cE)$. Any principal connection on $\cE\ra M$ is of the form $i\alpha$ where $\alpha\in\Omega^1(\cE)^{\gU(1)}$ and $i_{\partial_t}\alpha=1$. The space of  principal connections is an affine space over $\Omega^1(M)$, i.e. any principal connection is equal to $i(\alpha+q^\ast\theta)$ for a unique $\theta\in\Omega^1(M)$. Since $d\alpha\in\Omega^2(\cE)^{\gU(1)}$  and $i_{\partial_t}d\alpha=0$, there is a unique $\mathbf\Omega\in\Omega^2(M)$ such that $d\alpha=q^\ast\mathbf\Omega$. Obviously, $d\mathbf\Omega=0$ and by a remark above, it follows that the cohomology class $[\mathbf\Omega]\in H^2(M,\R)$ depends only on $\cE\ra M$.  The class $\frac{-1}{2\pi}[\mathbf\Omega]$ is integral and it is called the \textit{Chern 
class} of $\cE\ra M$.

\medskip

Let  $\cP\ra M$  be  a $\SO(n)$-principal bundle and $F:=\cP\times_{\SO(n)}\R^n$ be the associated vector bundle. A spin structure on  $F$ is a  $\Spin(n)$-principal bundle $\tilde\cP\ra M$ together with a fiber bundle map $\tilde\cP\ra\cP$ over the identity map $\Id_M$ on $M$ which is in each fiber compatible in an obvious way with  $\rho_n$. We also say that  $\tilde\cP\ra M$ is a \textit{lift} of $\cP\ra M$ to  spin structure. It is well known that $F$ admits a spin structure if and only if the second Stiefel-Whitney class $w_2(F)$ of $F$ is zero.

A spin$^c$ structure on $F$ is a $\Spin^c(n)$-principal bundle $\cP^c\ra M$ together with a bundle map $\cP^c\ra\cP$ over $\Id_M$ which is compatible with $\rho_n^c$. A spin$^c$ structure exists if, and only and if there  is an integral class $X\in H^2(M,\Z)$ such that $\rho(X)=w_2(F)$ where $\rho: H^2(M,\Z)\ra H^2(M,\Z_2)$ is the standard map  induced by  $\Z\ra\Z_2,\ i\mapsto i$ mod $2$. 

\medskip

Using the short exact sequences (\ref{second ses}) and (\ref{third ses}), we define
\begin{equation*}
\cE:=\cP^c\times_{\Spin^c(n)}\gU(1) \ \mathrm{and}\ 
\cP':=\cP^c\times_{\Spin^c(n)}(\gU(1)\times\SO(n)).
\end{equation*}
Then there is a commutative diagram 
\begin{equation}\label{com diagram with Spinc}
\xymatrix{\cP'\ar@{~>}[d]|{\SO(n)}&&\ar@{~>}|{\Z_2}[ll]\cP^c\ar@{=}[r]\ar@{~>}[d]|{\Spin(n)}&\cP^c\ar@{~>}|{\gU(1)}[r]\ar@{~>}[d]|{\Spin^c(n)}&\cP\ar@{~>}[d]|{\SO(n)}\\
\cE\ar@{=}[rr]&&\cE\ar@{~>}|{\gU(1)}[r]&M\ar@{=}[r]&M\\} 
\end{equation}
where $\xymatrix{\cG\ar@{~>}|G[r]&N}$ denotes a $G$-principal bundle $\cG\ra N$. It is easy to see that $\cP'\ra M$ is isomorphic to the fibered  product $\cE\times_M\cP\ra M$ and  that $\cE\times_M\cP\ra\cE$ is isomorphic to the pullback $q^\ast\cP\ra\cE$ where $q:\cE\ra M$.

\medskip

Let us now assume that $g_M$ is a Riemannian metric on $M$ and that $F=TM$. The Levi-Civita connection can be viewed as a principal connection $\omega$ on $\cP$. If $M$ admits a spin structure $\tilde\cP\ra M$ with a lift $t:\tilde\cP\ra\cP$, then  $t^\ast\omega$  is  the canonical spin principal connection induced by the Levi-Civita.  

On the other hand, it may happen that $M$ admits only a spin$^c$-structure  $\cP^c\ra M$. Let $\cE\ra M$ be the associated  circle bundle.  Then as we observed above,  $\cP^c\ra\cE\times_M\cP$ is a 2:1 covering and so a principal connection on $\cP^c\ra M$ is the pullback of a principal connection on $\cE\times_M\cP\ra M$. If $p_1:\cE\times_M\cP\ra\cE$ and $p_2:\cE\times_M\cP\ra\cP$ are the canonical projections and $i\alpha$ is a principal connection on $\cE\ra M$, then $p_1^\ast i\alpha\oplus p_2^\ast\omega$ is a principal connection on $\cE\times_M\cP\ra M$. We see that there is a family of natural principal connections on $\cP^c\ra M$ and each member of this family is determined by a principal connection on $\cE\ra M$. 

\section{Lie contact structure and 2-Dirac complex}\label{section Lie contact structure and Dirac complex}
A Lie contact structure is (see Section \ref{section Lie contact})  a  contact structure which is associated to a (unique) contact grading on $\laso(2,n+2)$. We will need in this article two types of Lie contact structures, an effective one (see Section \ref{section ef LCS}) and a non-effective one (see Section \ref{section non-ef LCS}). In order to properly define the the 2-Dirac complex introduced in \cite{S}, we will need  (see Section \ref{section Dirac complex}) the non-effective structure.  See \cite[Section 4.2.5]{CS} for more about the Lie contact structure.

\subsection{Lie contact structure}\label{section Lie contact}

Let $n\ge3$ be an integer, $\delta$ be the Kronecker delta, $\{\tau_1,\tau_2,\varepsilon_1,\ldots,\varepsilon_{n},\tau^1,\tau^2\}$ be the standard basis of $\R^{n+4}$ and  $h$ be the symmetric bilinear form on $\R^{n+4}$ determined by $h(\tau_i,\tau^j)=\delta_{ij},\ h(\tau_i,\varepsilon_\alpha)=h(\tau^i,\varepsilon_\alpha)=0,\ h(\varepsilon_\alpha,\varepsilon_\beta)=\delta_{\alpha\beta},\ i,j=1,2$ and $\alpha,\beta=1,\dots,n$. The  signature of $h$ is $(2,n+2)$ and we will for brevity denote  $(\R^{n+4},h)$  by $\R^{2,n+2}$.  The associated Lie algebra $\lag:=\laso(h)\cong\laso(2,n+2)$ is 
\begin{equation}\label{orthogonal Lie alg}
\Bigg\{\left(\begin{array}{ccc}
A&Z^T&W\\
X&B&-Z\\
Y&-X^T&-A^T
\end{array}\right)\Bigg|\ \begin{matrix}
A\in M_2(\R),\ Y,W\in A_2(\R),\\
B\in\laso(n),\ X,Z\in M_{n\times2}(\R)
\end{matrix}\Bigg\}.
\end{equation}
Then $\lag$ is a direct sum of the following five subspaces
\begin{eqnarray}\label{gradation on g}
&\lag_{-2}=\Bigg\{
\left(\begin{matrix}
0&0&0\\
0&0&0\\
\ast&0&0\\
\end{matrix}
\right)
\Bigg\},\
\lag_{-1}=\Bigg\{
\left(\begin{matrix}
0&0&0\\
\ast&0&0\\
0&\ast&0\\
\end{matrix}
\right)
\Bigg\},\ 
\lag_{0}=\Bigg\{
\left(\begin{matrix}
\ast&0&0\\
0&\ast&0\\
0&0&\ast\\
\end{matrix}
\right)
\Bigg\},\nonumber&\\ 
&\lag_{1}=\Bigg\{
\left(\begin{matrix}
0&\ast&0\\
0&0&\ast\\
0&0&0\\
\end{matrix}
\right)
\Bigg\}\ \ \mathrm{and}\ \
\lag_{2}=\Bigg\{
\left(\begin{matrix}
0&0&\ast\\
0&0&0\\
0&0&0\\
\end{matrix}
\right)
\Bigg\}.&
\end{eqnarray}
It is  straightforward to verify that:
\begin{enumerate}[I.]
\item $[\lag_i,\lag_j]\subset\lag_{i+j},\ i,j\in\Z$ where we agree that $\lag_i=\{0\}$ if $|i|>2$.\label{grading property}
\item  $\lag_{-1}$ generates $\lag_-:=\lag_{-2}\oplus\lag_{-1}$ as a Lie algebra.
\end{enumerate}
Hence, the direct sum is a $|2|$-grading on $\lag$. As $\dim(\lag_{-2})=1$ and since the Lie bracket $\Lambda^2\lag_{-1}\ra\lag_{-2}$ is non-degenerate, it follows that $\lag_-$ is a Heisenberg algebra and that the $|2|$-grading is a contact grading. 
\bigskip

Notice that $\lag_0\cong\lagl(2,\R)\oplus\laso(n)$. We will view any representation of $\lagl(2,\R)$ or $\laso(n)$ also as a $\lag_0$-module by letting the other factor act trivially. Then  there are isomorphisms of $\lag_0$-modules
\begin{equation}\label{g0 pieces in lap}
\lag_{-2}\cong\Lambda^{2}\mE^\ast,\  \lag_{-1}\cong\mE^\ast\otimes\mF,\  \lag_1\cong\mE\otimes\mF\ \mathrm{and}\ \lag_2\cong\Lambda^{2}\mE
\end{equation}
where $\mE$ and $\mF$  is the defining representation of $\lagl(2,\R)$ and  $\laso(n)$, respectively. By the Jacobi identity, the  Lie bracket $\Lambda^2\lag_{-1}\ra\lag_{-2}$ is $\lag_0$-equivariant. Using the isomorphisms from (\ref{g0 pieces in lap}), it is easy to see that there is  (up to constant) a unique $\lag_0$-equivariant map   
\begin{equation}\label{lie bracket on g-}
\Lambda^2\lag_{-1}\cong\Lambda^2(\mE^\ast\otimes\mF)\ra\Lambda^2\mE^\ast\otimes S^2\mF\ra\Lambda^2\mE^\ast\cong\lag_{-2}
\end{equation}
where the first map is the canonical projection and in the second map we  trace, using the standard inner product, over the second component. 

\subsection{Effective Lie contact structure}\label{section ef LCS}
Put $\lag^i:=\oplus_{j= i}^{2}\ \lag_j$, $\G:=\SO_o(2,n+2)$ where the subscript $o$  stands for the connected component of the identity element of the given group and  $\Ad$ be the adjoint action. We call
\begin{equation}\label{parabolic subgroup}
\gP:=\{g\in\G:\Ad(g)(\lag^i)\subset\lag^i,\ i=-2,\dots,2\}
\end{equation}
the \textit{parabolic subgroup}\footnote{Note that according to  \cite[Definition 3.1.3]{CS}, a parabolic subgroup corresponding to the contact grading on $\lag$ is  any subgroup $\gH$ such that   $\gP_o\subset\gH\subset\gP$} associated to the  contact grading on $\lag$ and 
\begin{equation}\label{Levi subgroup}
\G_0:=\{g\in\gP:\Ad(g)(\lag_i)\subset\lag_i,\ i=-2,\dots,2\}.
\end{equation}
the \textit{Levi subgroup} of $\gP$.
It is easy to see that $\G_0\cong\GL^+(2,\R)\times\SO(n)$ where $\GL^+(2,\R)=\{A\in\GL(2,\R):\ \det A>0\}$.

\medskip

Let $\cm$ be a manifold of dimension $2n+1$ with a contact distribution $H$. Then a \textit{Lie contact structure of type} $(\G,\gP)$ on $\cm$ with underlying contact structure $H$ is given by a $\G_0$-principal bundle $\cp_0:\ccG_0\ra\cm$ together with:
\begin{enumerate} 
 \item [(i)] $\theta_{-2}^\sharp\in\Omega^1(\ccG_0,\lag_{-2})^{\G_0}$ such that\footnote{By abuse of notation, if $\theta$ is a 1-form on a smooth manifold $N$ with values in a vector space $\mV$, then we denote by $\ker\theta$ the set of those tangent vectors $v\in T_xM$ such that $\theta_x(v)=0$.} $\ker(\theta_{-2})=T^{-1}\ccG_0$ and
\item [(ii)] $\theta_{-1}^\sharp\in\Gamma(L(T^{-1}\ccG_0,\lag_{-1}))^{\G_0}$ such that $\ker(\theta_{-1})=\ker(T\cp_0)$
 \end{enumerate}
where $T^{-1}\ccG_0:= (T\cp_0)^{-1}(H)$ and $L(T^{-1}\ccG_0,\lag_{-1})$ is the vector bundle  over $\ccG_0$ whose fiber over $\phi$ is the  space of  linear maps $T^{-1}_\phi\ccG_0\ra\lag_{-1}$ and the superscript $\G_0$ stands for equivariant forms, i.e. $(r^g)^\ast\theta_{-2}=\Ad(g)\circ\theta_{-2}$ where $r^g$ stands for the principal action by $g\in\G_0$ and similarly for $\theta_{-1}$. Moreover, there is a compatibility requirement between $\theta^\sharp:=(\theta^\sharp_{-2},\theta^\sharp_{-1})$ and the Levi form but, as we will not need it, we will not go into details, see \cite[Section 2.3]{CSII}.

Equivalently, this geometric structure is given by a pair of auxiliary vector bundles $E^\sharp$ and $F^\sharp$ of rank $2$ and $n$, respectively, together with  a bundle metric on $F^\sharp$ such that  $H\cong E^{\sharp\ast}\otimes F^\sharp$ and for each $\cx\in\cm$ the Levi form $\cL_{\cx}$ is invariant under the resulting action of the orthogonal group $\gO(F^\sharp_{\cx})$. 

\medskip

An infinitesimal symmetry of the Lie contact structure of type $(\G,\gP)$ is a  $\G_0$-equivariant vector field $\xi_0\in\mathfrak X(\ccG_0)^{\G_0}$ whose flow preserves $T^{-1}\ccG_0$ and  $\cL_{\xi_0}\theta^\sharp_i=\theta^\sharp_i,\ i=-1,-2$ where $\cL$ denotes the Lie derivative. Then $\xi_0$ is $\cp_0$-projetable, i.e.  $T\cp_0(\xi_0)\in\mathfrak X(\cm)$ is a well defined vector field. It is well known that $\xi_0$ is uniquely determined by $\xi$ and hence, we will often view $\xi_0$ as $\xi$ and vice versa without further comment. We call  $\xi_0$ a \textit{transversal infinitesimal symmetry} if for each $\cx\in\cm:\ \xi(\cx)\not\in H_{\cx}$ .

Assume that $\mV$ is an irreducible $\G_0$-module. The space of sections of $V^\sharp:=\ccG_0\times_{\G_0}\mV$ is canonically isomorphic to the space of smooth $\G_0$-equivariant $\mV$-valued functions on $\ccG_0$. Since $\cL_{\xi_0}f$ is an equivariant function provided that $f$ is, it follows that  $\xi_0$ induces a map $\cL_{\xi}:\Gamma(V^\sharp)\ra\Gamma(V^\sharp)$.

\medskip

Let us now consider the homogeneous model $\G/\gP$. The tangent bundle of $\G/\gP$ is isomorphic to the associated vector bundle $\G\times_{\gP}\lag/\lap$. As  $\lag^{-1}$ is a $\gP$-invariant subspace of $\lag$, it follows that $H=\G\times_{\gP}(\lag^{-1}/\lap)$ can be viewed as a  distribution over $\G/\gP$ of co-dimension 1 which  is well known to be contact. Moreover, we can take $\ccG_0=\G/\exp(\lag^1)$ and use the Maurer-Cartan form on $\G$ to construct the forms $\theta^\sharp_i,\ i=-2,-1$. This gives the homogeneous Lie contact structure of type $(\G,\gP)$ on the homogeneous model.

\smallskip

Let us now be more explicit.
It is easy to see that $\gP$  is the stabilizer of the totally isotropic 2-plane $[\tau_1,\tau_2]$ where we denote by square brackets the linear span of the given vectors. As $\G$ acts transitively on the  Grassmannian $\pgr$ of totally isotropic 2-dimensional subspaces in $\R^{2,n+2}$,  we may view $\G/\gP$ as $\pgr$.  

In order to view $\G/\gP$ as a Stiefel variety, choose subspaces $\R^2$ and $\R^{n+2}$ of $\R^{2,n+2}$ with bases $\{f_1,f_2\}$ and $\{e_1,\dots,e_{n+2}\}$, respectively, such that:
\begin{enumerate}
\item[(i)] $\R^{2,n+2}=\R^2\oplus\R^{n+2}$,
\item[(ii)]  $h|_{\R^2}$  is  negative definite  with orthonormal basis $\{f_1,f_2\}$ and
\item[(iii)] $h|_{\R^{n+2}}$ is  positive definite with orthonormal basis $\{e_1,\dots,e_{n+2}\}$. 
\end{enumerate}
Fix  $\cx\in\pgr$. Then  there are unique $v_1,v_2\in\R^{n+2}$ such that $\cx=[f_1+v_1,f_2+v_2]$. Since $\cx$ is totally isotropic, it follows that $(v_1,v_2)$ is an orthonormal 2-frame. On the other hand, any such orthonormal 2-frame determines a maximal totally isotropic  subspace in $\R^{2,n+2}$. We see that $\pgr$ is isomorphic to the Stiefel manifold $\sv$.

The stabilizer $\cgK$ of $\R^2$ and $\R^{n+2}$ in $\G$ is (with respect to the new basis of $\R^{2,n+2}$) the standard block diagonal subgroup $\SO(2)\times\SO(n+2)$. Obviously, $\cgK$ is a maximal compact subgroup of $\G$ and the canonical $\cgK$-action on $\sv$ is: $(A,B).(v_1|v_2):=B.(v_1|v_2).A^{-1}$ where $A\in\SO(2),\ B\in\SO(n+2)$ and we view $(v_1,v_2)$ as the corresponding matrix $(v_1|v_2)\in M_{(n+2)\times2}(\R)$.
It is easy to see that the action is transitive and  that 
\begin{equation}\label{stabilizer of cx0}
\cgH=\Bigg\{\left(
\begin{matrix}
A&0&0\\
0&A&0\\
0&0&B\\
\end{matrix}
\right)\Bigg|A\in\SO(2),B\in\SO(n)\Bigg\}
\end{equation}
is the stabilizer of $\cx_0:=(e_{1},e_{2})\in\sv$. Hence, $\sv\cong\cgK/\cgH$ and it is straightforward to verify that the map $\pgr\ra\sv$ defined above is a $\cgK$-equivariant diffeomorphism.

\smallskip

We can view a tangent vector at $(v_1,v_2)\in\sv$ as a skew-symmetric map $\psi:\R^{n+2}\ra\R^{n+2}$, i.e. $h|_{\R^{n+2}}(\psi(u),v)=-h|_{\R^{n+2}}(u,\psi(v)), \ u,v\in\R^{n+2}$.
\begin{lemma}\label{lemma contact distribution over sv}
With the notation introduced above, $\psi\in H_{(v_1,v_2)}$  if and only if  $\psi([v_1,v_2])\subset([v_1,v_2])^\bot$.
\end{lemma}
\begin{proof}
Let $G_2(\R^{n+4})$ be the Grassmannian of 2-dimensional subspaces in $\R^{n+4}$ and $\cx\in G_2(\R^{n+4})$. Then $T_{\cx} G_2(\R^{n+4})$ is canonically isomorphic to the space of linear maps $\psi:\cx\ra\R^{n+4}/\cx$. If $\cx\in\pgr$, then $\psi\in T_{(v_1,v_2)}\pgr$ if and only if the bilinear form on $\cx: (u,v)\mapsto h(u,\psi(v))$  is skew-symmetric. Moreover, $\psi\in H_{\cx}$ if and only if $\im(\psi)\subset(\cx)^\bot$ (with respect to $h$). If $\cx=[f_1+v_1,f_2+v_2]$, then  $\R^{n+4}/\cx\cong\R^{n+2}=[v_1,v_2]\oplus[v_1,v_2]^\bot$ and without loss of generality we may assume that  $\im(\psi)\subset\R^{n+2}$. Then the claim easily follows.
\end{proof}

\subsection{Non-effective Lie contact structure}\label{section non-ef LCS}

Let $\tilde\G:=\Spin_o(2,n+2)$ and $\rho_{2,n+2}:\tilde\G\ra\G$ be the standard 2:1 covering. We define the parabolic subgroup $\tilde\gP$ of $\tilde\G$ corresponding to the contact grading on $\lag$ as before for $\G$ and denote by $\tilde\G_0$ its Levi subgroup.
It is clear that $\G/\gP\cong\tilde\G/\tilde\gP$ and  that  $\tcgK:=\rho_{2,n+2}^{-1}(\cgK)$ is a maximal compact subgroup of $\tilde\G$ which, see Lemma \ref{lemma inclusion of spinc inside spin}, is isomorphic to $\Spin^c(n+2)$.  

\begin{lemma}\label{lemma groups in covering}
The stabilizer $\tcgH$ of $\cx_0$ inside $\tcgK$ is isomorphic to $\SO(2)\times\Spin(n)$ and  $\rho_{2,n+2}|_{\tcgH}:\tcgH\ra\cgH$  is equal to $\Id_{\SO(2)}\times\rho_n$.
\end{lemma}
\begin{proof}
Recall (\ref{stabilizer of cx0}).  It follows that $\tcgH$ is isomorphic to the quotient of the group $\{(e^{it},e^{is},a)|\ e^{it},e^{is}\in\gU(1),\ a\in\Spin(n):\ e^{it}=\pm e^{is}\}$
by the normal subgroup generated by $\{(-1,-1,1),(-1,1,-1)\}$ such that, if we denote by $\langle e^{it},e^{is},a\rangle$ the class of $(e^{it},e^{is},a)$ in the quotient, the map $\rho_{2,n+2}|_{\tcgH}:\tcgH\ra\cgH=\SO(2)\times\SO(n)$ is $\langle e^{it},e^{is},a\rangle\mapsto(\rho_2(e^{is}),\rho_n(a))$. 
Consider  $\tcgH\ra\SO(2)\times\Spin(n),\ \langle e^{it},e^{is},a\rangle\mapsto(\rho_2(e^{is}),e^{i(t-s)}a)$. Then it is straightforward to verify that the map is well defined, bijective and that it is a homomorphism.  The inverse homomorphism is  $\SO(2)\times\Spin(n)\ra\tcgH,\ (e^{is},a)\mapsto\langle e^{\frac{is}{2}},e^{\frac{is}{2}},a\rangle$. Composing this with $\rho_{2,n+2}|_{\tilde\gH}$, the last claim follows.
\end{proof}

It is clear that we can choose the subspaces $\R^2$ and $\R^{n+2}$ and the bases $\{f_1,f_2\}$ and $\{e_1,\dots,e_{n+2}\}$ so that $\cgH$ and $\tcgH$ is a maximal compact subgroup of $\G_0$ and $\tilde\G_0$, respectively. By  Lemma \ref{lemma groups in covering}, it follows that 
$\tilde\G_0\cong\GL^+(2,\R)\times\Spin(n)$ and that  $\rho_{2,n+2}|_{\tilde\G_0}:\tilde\G_0\ra\G_0$ is equal to $\Id_{\GL^+(2,\R)}\times\rho_n$.

\medskip

A \textit{Lie contact structure of type} $(\tilde\G,\tilde\gP)$ on a manifold $\cm$ of dimension $2n+1$ with a contact structure $H$ is given by a Lie contact structure $(\ccG_0,\cm,\tcp_0,\theta^\sharp)$ of type $(\G,\gP)$ as in Section \ref{section ef LCS} and a lift $\tcp_0:\tccG_0\ra\cm$  of $\ccG_0\ra\cm$ to $\tilde\G_0$-structure.   We denote by $\tilde\theta^\sharp$  the pullback of $\theta^\sharp$ to $\tccG_0$.

\subsection{2-Dirac complex}\label{section Dirac complex}
The subspace of diagonal matrices is a maximal commutative subalgebra of $M(2,\R)=\lagl(2,\R)$. Hence, a weight of $\lagl(2,\R)$, and thus also of $\GL^+(2,\R)$, can be given by a pair of real numbers. We denote by  $\mE_\lambda$ a real irreducible $\GL^+(2,\R)$-module with highest weight $\lambda$ and by  $\Sp$ the complex spinor representation of $\Spin(n)$. Then  $\mV_\lambda:=\mE_{\lambda}\otimes\Sp$ is an irreducible $\tilde\G_0$-module. We call (\ref{k-Dirac sequence}) the \textbf{2-Dirac complex}.

\begin{thm}\label{thm 2-Dirac complex}
Let $(\tccG_0,\cm,\tcp_0,\tilde\theta^\sharp)$ be a Lie contact structure of type $(\tilde\G,\tilde\gP)$. Let
$\mV_{\lambda_i},\ i=0,1,2,3$ be the irreducible $\tilde\G_0$-modules associated to
\begin{align}\label{weight in Dirac complex}
\lambda_0&=\frac{1}{2}(n-1,n-1),\ \lambda_1=\frac{1}{2}(n+1,n-1),\\
\lambda_2&=\frac{1}{2}(n+3,n+1)\ \mathrm{and}\ \lambda_3=\frac{1}{2}(n+3,n+3)\nonumber
\end{align}
as explained above and put $V^\sharp_i:=\tccG_0\times_{\tilde\G_0}\mV_{\lambda_i}$.

Then there is a sequence of natural linear differential operators 
\begin{equation}\label{k-Dirac sequence}
\Gamma(V^\sharp_0)\ra\Gamma(V^\sharp_1)\ra\Gamma(V^\sharp_2)\ra\Gamma(V^\sharp_3)
\end{equation}
of order 1,2,1 respectively. The sequence is on the homogeneous model a complex.
\end{thm}
\begin{proof}
The operators were (see \cite{S}) constructed using the machinery of the splitting operators and the curved Casimir operator. As these are linear and natural differential operators, the operators in  (\ref{k-Dirac sequence}) are also linear and natural. Moreover, it was verified directly that on the homogeneous space it is a complex. 
\end{proof}

\begin{remark}\label{remark bundles in Dirac complex}
1) Notice that $\mE_{\lambda_0}\cong\mE_{\lambda_3}\cong\R$ and $\mE_{\lambda_1}\cong\mE_{\lambda_2}\cong\R^2$ as $\SL(2,\R)$-modules. 

2) The 2-Dirac complex is not a BGG-sequence. However, it (see \cite{TS}) arises as the direct image of a relative BGG complex and so it fits into the scheme of the Penrose transform. Using this approach, one can reprove that on the homogeneous model the $2$-Dirac complex is  a complex of differential operators.  
\end{remark}

\section{PCS-structures of Grassmannian type}\label{section PACS structures}
We will  recall (see Section \ref{section PACS structures of GT}) the definition of the PCS-structure of Grassmannian type. Original results are given in Sections \ref{section PCS-quotient} and \ref{section descending natural differential operators}.

\subsection{PCS structure of Grassmannian type}\label{section PACS structures of GT}
Let $M$ be a manifold of dimension $2n\ge6$. We call a line subbundle $\ell$ of  $\Lambda^2T^\ast M$ an \textit{almost conformally symplectic structure} (or $acs$-structure for short) if each $\omega\in\ell$ is either zero or a non-degenerate 2-form. We call $\ell$ a \textit{conformally symplectic structure} (or $cs$-structure for short) if in addition, each $x\in M$ has an open neighborhood $U$ with an everywhere non-zero section  $\sigma:U\ra\ell|_U$  which is closed as a 2-form. 

Recall (see  \cite[Section 4.1.3]{CS}) that an  \textit{almost Grassmannian structure of type} $(p,q)$ on a smooth manifold  $N$ of dimension $p.q$ is given by two auxiliary vector bundles $E$ and $F$ of rank $p$ and $q$, respectively, together with an isomorphism $TN\ra E^\ast\otimes F$ and a fixed trivialization of $\Lambda^pE\otimes\Lambda^qF$.   In this paper we will always assume that $E,F$ are oriented.
It is easy to see that the structure is  equivalent to a reduction of the frame bundle of $N$ to structure group $\G_{gr}:=\{(A,B)\in\GL^+(p,\R)\times\GL^+(q,\R):\ \det(A)=\det(B)^{-1}\}$ so that  $E$ and $F$ are associated to the standard representations  of $\G_{gr}$ on $\R^p$ and $\R^q$, respectively. Notice that   $\Lambda^2T^\ast N\cong\Lambda^2E\otimes S^2 F^\ast\oplus S^2E\otimes \Lambda^2F^\ast$.

\begin{df}\label{df PCS structure of Grassmannian type}
A parabolic almost conformally symplectic structure (or PACS-structure for short) of Grassmannian type on a manifold $M$ of dimension $2n\ge6$ is  given by a Grassmannian structure of type $(2,n)$ with auxiliary oriented bundles $E$ and $F$  together with an almost conformally symplectic structure $\ell\subset\Lambda^2T^\ast M$ which is Hermitian in the Grassmannian sense, i.e. $\ell$ is contained in $\Lambda^2E\otimes S^2F^\ast$. We call the structure a parabolic conformally symplectic structure (or PCS-structure for short) of Grassmannian type if $\ell$ is conformally symplectic.
\end{df}

As $\ell$ and $\Lambda^2E$ are line bundles and $F$ is oriented, we see that a Hermitian $acs$-structure $\ell$ determines a conformal class of bundle metrics on $F$ and conversely, a conformal class of metrics on $F$ determines a Hermitian $acs$-structure. The signature of any metric from the class is an invariant of the $acs$-structure. Here we will consider only the positive definite case.
\medskip

A PCS-structure of Grassmannian type on $M$ is equivalent to a certain $G$-structure $p_0:\cG_0\ra M$. Any such structure carries a tautological form $\theta$, called the \textit{soldering form}, which is an equivariant $\lag_{-1}$-valued 1-form on $\cG_0$ whose pointwise kernel is the vertical subbundle.

\begin{lemma}\label{lemma G0 structure}
 A PACS-structure of Grassmannian type is equivalent to a $\G_0$-structure $(\cG_0,M,p_0,\theta)$. 
\end{lemma}
\begin{proof}
By above, a conformal class of positive definite metrics on $F$ is equivalent to a reduction of  the structure group from $\G_{gr}$ to $\G_{gr}\cap(\GL^+(2,\R)\times\CSO(n))$. Since $(A,B)\mapsto(A,\det(A)^{\frac{1}{n}}.B)$ is an isomorphism between the latter group and $\G_0$, the claim follows.
\end{proof}

There is (see \cite{CSI}) a canonical linear connection on $TM$ which is compatible with a given PACS-structure (of any type). This linear connection is determined by a normalization condition on its torsion, see \cite[Corollary 4.3]{CSI}. It turns out that the intrinsic torsion splits into two components. The first component is an invariant of the underlying $acs$-structure $\ell$, and it vanishes if and only if $\ell$ is a $cs$-structure. The other component is called a \textit{harmonic torsion}. If the PACS-structure is of Grassmannian type, then there is an isomorphism $\Lambda^2T^\ast M\otimes TM\cong(\Lambda^2E\otimes S^2F^\ast\oplus S^2E\otimes\Lambda^2 F^\ast)\otimes E^\ast\otimes F^\ast$. If it is a PCS-structure, then  the  canonical connection is pinned down by the requirement that its torsion is a section of  
\begin{equation}
\Lambda^2E\otimes E^\ast\otimes (S^3F^\ast)_0\oplus (S^2E\otimes E^\ast)_0\otimes(\Lambda^2F^\ast\otimes F)_0
\end{equation}
where  the subscript 0 denotes the trace-free part.

\medskip

Let us now recall   \cite[Definition 2.4]{CSII}.

\begin{df}\label{df PCS quotient}
Let  $(\ccG_0,\cm,\cp_0,\theta^\sharp)$ be a Lie contact structure of type $(\G,\gP)$ with a transversal infinitesimal symmetry $\xi_0\in\mathfrak X(\ccG_0)^{\G_0}$. Then a PCS-quotient by $\xi_0$ is  a PCS-structure $(\cG_0,M,p_0,\theta)$ of Grassmannian type together with a morphism $q_0:\ccG_0\ra\cG_0$ of  $\G_0$-principal bundles such that the  following holds:
\begin{enumerate}
\item $q_0$ is surjective with connected fibers.
\item For each $u^\sharp\in\ccG_0$ the kernel of $T_{u^\sharp}q_0$ is spanned by $\xi_0(u^\sharp)$.
\item The restriction of $q_0^\ast\theta$ to $T^{-1}\ccG_0$ coincides with  $\theta^\sharp_{-1}$.
\end{enumerate}
\end{df}

We will be interested here in PCS-quotients whose fibers are circles. Recall Section \ref{section spinc structure} for the definition of the  Chern class $c_1(\cm)$ associated to a circle bundle $\cm\ra M$.

\begin{lemma}\label{lemma compatible circle bundle}
Let $(\cG_0,M,p_0,\theta)$ be a PCS-structure of Grassmannian type such that the underlying $cs$-structure $\ell$ is trivialized by a global symplectic form $\mathbf\Omega$. Assume that  $\frac{-1}{2\pi}[\mathbf\Omega]$  is integral and that $q:\cm\ra M$ is a circle bundle whose Chern class is equal to this integral class. Then
\begin{enumerate}
\item[(i)] the unique principal connection $i\alpha\in\Omega^1(M^\sharp)^{\gU(1)}$ determined by $q^\ast\mathbf\Omega=d\alpha$ is a contact form,
\item[(ii)] there is a  parabolic contact structure $(\ccG_0,\cm,\cp_0,\theta^\sharp)$ of type $(\G,\gP)$ with the underlying contact structure $H:=\ker(\alpha)$  such that the Reeb field $\xi$ associated to $\alpha$ is a transversal infinitesimal symmetry of the Lie contact structure and 
\item[(iii)] there is a PCS-quotient $q_0:\ccG_0\ra\cG_0$ by $\xi$.
\end{enumerate}
\end{lemma}

\begin{proof}
(i) By assumptions, $\mathbf\Omega^{\wedge^n}$ is a volume form on $M$ and so $\alpha\wedge q^\ast\mathbf\Omega^{\wedge^n}$ is a volume form on $\cm$.
(ii) and (iii) As $\alpha$ is a contact form and $d\alpha=q^\ast\mathbf\Omega$, it follows that $q:\cm\ra M$ is (see \cite[Section 2.4]{CSII}) a reduction by the transverse infinitesimal contactomorphism $\xi$. Then (ii) and (iii) follow from \cite[Theorem 2.7]{CSII}. 
\end{proof}

We will call $\alpha$ from Lemma \ref{lemma compatible circle bundle} a \textit{compatible contact form}. 

\begin{remark}\label{remark construction of lift}
Let us also recall that in Lemma \ref{lemma compatible circle bundle} one can take as  $\ccG_0$  the pullback bundle $q^\ast\cG_0$, as $\theta^\sharp_{-1}$  the restriction of $q^\ast_0\theta$ and then there is a unique way how to define $\theta_{-2}^\sharp$ so that $\theta^\sharp=(\theta^\sharp_{-1},\theta^\sharp_{-2})$ together with $\cp_0:\ccG_0\ra\cm$ give a Lie contact structure of type $(\G,\gP)$.
\end{remark}

\subsection{Compatible $\G_0^c$-structures and Lie contact structures.}
\label{section PCS-quotient}

Put $\G_0^c:=\GL^+(2,\R)\times\Spin^c(n)$. Then there are two Lie group homomorphisms and a short exact sequence
\begin{align}
 \G_0^c\ra\G_0 \ \ \mathrm{and}\ \ \G_0^c\ra\gU(1)\ \ \mathrm{and}\ \ 
 0\ra\Z_2\ra\G_0^c\ra\gU(1)\times\G_0\ra0
\end{align}
induced by  $\rho_n^c:\Spin^c(n)\ra\SO(n)$ and $\varsigma_n:\Spin^c(n)\ra\gU(1)$  and (\ref{third ses}), respectively. 

Assume that $\cG_0^c\ra M$ is a $\G_0^c$-principal bundle which lifts a $\G_0$-principal bundle $\cG_0\ra M$. Then  $\cm:=\cG_0^c\times_{\G_0^c}\gU(1)$ is the total space of a $\gU(1)$-principal bundle over $M$ which we call the associated \textit{determinant circle bundle}. The subgroup $\Z_2$ is normal and the quotient space $\cG_0':=\cG_0^c/\Z_2$ is the total space of a $\G_0\times\gU(1)$-principal bundle over $M$. The canonical projection $\cG_0'\ra\cm$ is a $\G_0$-principal bundle which is isomorphic to the pullback bundle $q^\ast\cG_0\ra\cm$. All together there is a commutative diagram 
\begin{equation}\label{com diagram with G0c}
\xymatrix{q^\ast\cG_0\cong\cG_0'\ar@{~>}[d]|{\G_0}&&\ar@{~>}|{\Z_2}[ll]\cG_0^c\ar@{=}[r]\ar@{~>}[d]|{\tilde G_0}&\cG_0^c\ar@{~>}|{\gU(1)}[r]\ar@{~>}[d]|{\G_0^c}&\cG_0\ar@{~>}[d]|{\G_0}\\
\cm\ar@{=}[rr]&&\cm\ar@{~>}|{\gU(1)}[r]&M\ar@{=}[r]&M\\} 
\end{equation}
where we use same notation as in (\ref{com diagram with Spinc}).

\medskip

Assume now that  $\cG_0\ra M$ is the principal bundle of a PCS-structure of Grassmannian type.
A principal connection on $\cG_0^c\ra M$ which lifts the canonical principal connection on  $\cG_0\ra M$ is determined by a choice of a principal connection on  $\cm\ra M$. If the principal connection on $\cm\ra M$ is compatible, as defined at the end of Section \ref{section PACS structures of GT}, with the underlying $cs$-structure, then  there is an additional geometric structure on $\cm$.

\begin{prop}\label{thm compatible G0c structure}
Let $(\cG_0,M,p_0,\theta)$ be a PCS-structure of Grassmannian type with a lift  $p_0^c:\cG_0^c\ra M$ of $\cG_0\ra M$ to a $\G_0^c$-structure. Let $q:\cm\ra M$ be the associated determinant circle bundle. Suppose that the underlying $cs$-structure $\ell$ is trivialized by a global symplectic form $\mathbf\Omega$ such that $\frac{-1}{2\pi}[\mathbf\Omega]=c_1(\cm)$. Let $\alpha$ be the  compatible contact form as in Lemma \ref{lemma compatible circle bundle}.

Then $\tcp_0:\cG_0^c\ra\cm$ is the principal bundle of a Lie contact structure of type $(\tilde\G,\tilde\gP)$ with the underlying contact distribution $H:=\ker(\alpha)$ and the Reeb field $\xi$ associated to $\alpha$ is a transversal infinitesimal symmetry of the Lie contact structure.
\end{prop}
\begin{proof}
As we have seen above, there is the $\G_0$-principal bundle $\cG_0'\ra\cm$ underlying $\cG_0^c\ra\cm$ and conversely, $\cG_0^c\ra\cm$ is a lift of the $\G_0$-principal bundle to structure group $\tilde\G_0$. Hence, it is enough to show that $\cG_0'\ra\cm$ is the principal bundle of a Lie contact structure of type $(\G,\gP)$ with underlying contact distribution $H=\ker(\alpha)$ and that $\xi$ is an transversal infinitesimal symmetry. But this follows from Lemma \ref{lemma compatible circle bundle} and Remark \ref{remark construction of lift} since we know that the bundle $\cG_0'\ra\cm$ is isomorphic to $q^\ast\cG_0\ra\cm$.
\end{proof}

\subsection{Descending linear and natural differential operators}\label{section descending natural differential operators}
As we have seen in Section \ref{section ef LCS}, an infinitesimal symmetry of the Lie contact structure has a canonical action, which we denote by $\cL_{\xi}$, on the space of sections of any associated vector bundle. We call a  differential operator \textit{natural} or \textit{invariant} if it intertwines with the actions of any infinitesimal symmetry on the corresponding spaces of sections. See also \cite[Sections 2.2 and 2.5]{CSIII}. 

\begin{thm}\label{thm descending}
Let $p_0:\cG_0\ra M$, $\tcp_0:\cG_0^c\ra\cm$ and $\xi$ be as in Proposition \ref{thm compatible G0c structure}.

(i) Let  $(\mV,\varrho)$ be a $\tilde\G_0$-representation  and  $(\mV,\varrho^c)$ be the associated $\G_0^c$-module as in (\ref{spinc representation on spinors}). Put $V:=\cG^c_0\times_{\G_0^c}\mV$ and $V^\sharp:=\cG^c_0\times_{\tilde\G_0}\mV$. Then there is a canonical linear isomorphism 
\begin{equation}\label{invariant differential equation}
\varphi_\mV: \Gamma(V)\ra\{v\in\Gamma(V^\sharp):\ \cL_{\xi}v=-i.v\}.
\end{equation}

(ii) Let $(\mW,\vartheta)$ be another $\tilde\G_0$-module, $(\mW,\vartheta^c)$ be the induced representation of $\G_0^c$ and put $W^\sharp:=\cG_0^c\times_{\tilde\G_0}\mW$. Suppose that $D^\sharp:\Gamma(V^\sharp)\ra\Gamma(W^\sharp)$ is a linear differential operator which is natural to the Lie contact structure of type $(\tilde\G,\tilde\gP)$. 
Then there is a unique linear differential operator $D:\Gamma(V)\ra\Gamma(W)$ such that the following diagram commutes:
\begin{equation}
 \xymatrix{\Gamma(V^\sharp)\ar[r]^{D^\sharp}&\Gamma(W^\sharp)\\
 \Gamma(V)\ar[u]^{\varphi_\mV}\ar[r]^{D}&\Gamma(W)\ar[u]^{\varphi_\mW}.}
\end{equation}
\end{thm}
\begin{proof}
(i) The space $\Gamma(V)$ is canonically isomorphic to the space  of smooth $\mV$-valued $\G_0^c$-equivariant functions with domain $\cG_0^c$. As $\G_0^c$ is generated by two commuting subgroups $\tilde\G_0$ and $\gU(1)$, a function $f:\cG_0^c\ra\mV$ is $\G_0^c$-equivariant if and only if it is equivariant with respect to $\tilde\G_0$ and $\gU(1)$. As $\gU(1)$ is connected, $f$ is $\gU(1)$-equivariant if and only if it transforms accordingly with respect to the fundamental vector field $\zeta$ associated to  $(i,0)\in\lau(1)\oplus\lag_0=\lag^c_0$. This means that  $\zeta(f)=\frac{d}{dt}|_{t=0}e^{-it}.f=-i.f$. Now $\zeta$ is uniquely determined by $(\tcp_0)^\ast\alpha(\zeta)=1$ and the fact that it lies in the vertical distribution of $\cG_0^c\ra\cG_0$. On the other hand, $\alpha$ is a contact form with Reeb field $\xi$ and so $(\tcp_0)^\ast\alpha(\tilde\xi_0)=\alpha(\xi)=1$. As also $\tilde\xi_0$ projects to zero under $\cG_0^c\ra\cG_0$, we find that $\zeta=\tilde\xi_0$ and so the claim follows from the 
definition of $\cL_\xi$.

(ii) As $D^\sharp$ is a natural differential operator and $\xi$ is an infinitesimal symmetry, $D^\sharp$ commutes with $\cL_\xi$. So if $v\in\Gamma(V^\sharp)$ satisfies $\cL_{\xi}v=-iv$, then  $\cL_{\xi}D^\sharp v=D^\sharp\cL_{\xi}v=D^\sharp(-iv)=-iD^\sharp v$. The existence of $D$ then follows. Arguing as in  \cite[Theorem 2.4]{CSIII}, one can show that $D$ is a differential operator.   
\end{proof}

\begin{remark} \label{remark not pcs quotient}
Suppose that we are given a PCS-quotient $q_0:\ccG_0\ra\cG_0$ by a transversal infinitesimal symmetry $\xi_0$ as in Definition \ref{df PCS quotient}. Let $\mV$ be an irreducible $\G_0$-module and put $V^\sharp:=\ccG_0\times_{\G_0}\mV$ and $V:=\cG_0\times_{\G_0}\mV$. Then (see \cite[Lemma 2.4]{CSIII}) there is a canonical isomorphism $\Gamma(V)\ra\{v\in\Gamma(V^\sharp):\ \cL_\xi v=0\}$, i.e. we have  identified $\Gamma(V)$ with the subspace of $\Gamma(V^\sharp)$ of those sections that satisfy the invariant differential equation $\cL_\xi v=0$. Let $D^\sharp:\Gamma(V^\sharp)\ra\Gamma(W^\sharp)$ be a natural and linear differential operator as above.  Then by the same argument as in the proof of Theorem \ref{thm descending}, the operator $D^\sharp$ descends to a linear differential operator $D:\Gamma(V)\ra\Gamma(W)$ which is natural to the PCS-structure.

The assumptions in Theorem \ref{thm descending} lead to a different invariant differential equation, namely the one given in (\ref{invariant differential equation}). As we observed in the proof, this differential equation comes from the fact that $\tilde\xi_0$ is a transversal infinitesimal symmetry of the Lie contact structure and at the same time it is the fundamental vector field corresponding to a generator of the center of $\Spin^c(n)$. This is the main difference between the spin$^c$ case studied in this article and  the scheme presented in the series \cite{CSI}, \cite{CSII} and \cite{CSIII} as,  when one deals with PCS-quotients,  the transversal infinitesimal symmetry is not "visible" downstairs on the PCS-structure but only upstairs on the parabolic contact structure.  Nevertheless, as we argued above, the main idea of descending natural and linear differential operators goes through also with in the spin$^c$ case.
\end{remark}

\section{Elliptic complex on $\rgr$}\label{section final}
In Section \ref{section final} we will prove the main result of this article.
We will show (see Proposition \ref{thm pcs quotient sc to rgr}) that there is a homogeneous PCS-structure of Grassmannian type on  the Grassmannian $\rgr$ of oriented 2-planes in $\R^{n+2}$ which descends from the homogeneous Lie contact structure of type $(\G,\gP)$ on $\sv$.
Application of Theorem \ref{thm descending} to the 2-Dirac complex induces (see Theorem \ref{thm descended complex}) the complex  of invariant differential operators on $\rgr$ that starts with  the 2-Dirac operator associated to a $\G_0^c$-structure as outlined in Introduction. The descended complex    is (see Theorem \ref{thm ellipticity and index}) elliptic and  its index is zero.

\subsection{Homogeneous PCS-structure on $\rgr$.}
There is a canonical projection $q:\sv\ra\rgr$ which sends an orthonormal 2-frame $(v_1,v_2)$ to its oriented span $[v_1,v_2]^+$. Recall Section \ref{section ef LCS} that $\sv$ is the homogeneous space of the Lie contact structure of type $(\G,\gP)$. Finally, we will need that  $\gK:=\SO(n+2)$ has  a canonical action on $\rgr$. 

\begin{prop}\label{thm pcs quotient sc to rgr}
Let $n\ge3$ and $(\ccG_0,\sv,\cp_0,\theta^\sharp)$ be the homogeneous Lie contact structure of type $(\G,\gP)$ and $\xi$ be an infinitesimal generator of the left action of the center $Z(\cgK)$ of $\cgK$ on $\sv$. 

Then $\xi$ is a transversal infinitesimal symmetry and there is a homogeneous $\gK$-invariant PCS-structure $(\cG_0,\rgr,p_0,\theta)$ of Grassmannian type together with a PCS-quotient $q_0:\ccG_0\ra\cG_0$ by  $\xi$ which covers the map $q:\sv\ra\rgr$ given above.

The underlying Grassmannian structure is the standard one and the $cs$-structure corresponds to the ray of bundle metrics on $F$ that contains the metric induced by the standard inner product on $\R^{n+2}$.
\end{prop}
\begin{proof}
As $\xi$ is an  infinitesimal generator of the left action, it is an infinitesimal symmetry. By the action of $\cgK$ on $\sv$ given in  Section \ref{section ef LCS}, it follows that $\xi((v_1,v_2))\in T_{(v_1,v_2)}\sv$ corresponds to an infinitesimal rotation in the plane $[v_1,v_2]$. By Lemma \ref{lemma contact distribution over sv},  $\xi\not\in H_{(v_1,v_2)}$ and thus, $\xi$ is a transversal infinitesimal symmetry.

By definition, the fibers of  $q$ coincide with the orbits of the left action of $Z(\cgK)$. These orbits are circles that can be also viewed as leaves for the distribution spanned by $\xi$. The map $\cp_0$  intertwines the action of $Z(\cgK)$ on $\ccG_0$ and $\sv$.  Since $Z(\cgK)$ acts transitively and freely on each leaf on $\sv$, the restriction of $\cp_0$ to each leaf on $\ccG_0$ is a diffeomorphism onto the underlying leaf on $\sv$. By  \cite[Theorem 2.5]{CSII}, there is a PCS-structure of Grassmannian type $(\cG_0,\rgr,p_0,\theta)$  and a PCS-quotient $q_0:\ccG_0\ra\cG_0$ by the symmetry $\xi$ which covers $q$. Moreover, the left action by any element from $\cgK$ is a diffeomorphism $\sv\ra\sv$ which maps leaves to leaves.  Hence, the action descends to $\rgr$ and it factorizes to the standard action of  $\gK=\cgK/Z(\cgK)$ on $\rgr$ by symmetries of the PCS-structure.

As $\cG_0=Z(\cgK)\setminus \ccG_0$ and $\ccG_0\cong\cgK\times_{\cgH}\G_0$, we see that $\cG_0\cong\gK\times_\gH\G_0$, i.e. it is an extension of the principal bundle $\gH\ra\gK\ra\rgr$ where $\gH$ is the stabilizer of the oriented 2-plane $x_0:=[e_1,e_2]^+$. From this the last claim readily follows.
\end{proof} 

Let us also mention that the unique compatible linear connection is the Levi-Civita connection of the symmetric Riemannian structure on $\rgr$. This is a special symplectic connection with holonomy contained in $\SL(2,\R)\times\SO(n)$, see \cite{CaS} for details. 

\begin{lemma}
The principal bundle $\cG_0\ra\rgr$ of the homogeneous PCS-structure of Grassmannian type  does not lift to a $\tilde\G_0$-structure.
\end{lemma}
\begin{proof}
It is well known that for the tautological vector bundles over $\rgr$: $w_2(F)=w_2(E)=\rho(e(E))$ where $e(E)$ is the Euler class. As $e(E)$ is the generator of $H^2(\rgr,\Z)\cong\Z$, it follows that $F$ does have a spin structure. The claim then easily follows from the discussion in Section  \ref{section PACS structures of GT}.
\end{proof}

We see that, as  mentioned in Introduction, there is no hope to get the 2-Dirac operator which is associated to  a $\tilde\G_0$-structure which lifts $\cG_0\ra\rgr$.

\subsection{Elliptic complex}\label{section elliptic complex}
Recall Section \ref{section non-ef LCS} that $\tcgK\cong\Spin^c(n+2)$ is a maximal compact subgroup of $\tilde\G$. The Stiefel variety $\sv$ is diffeomorphic to  $\tcgK/\tcgH$ where $\tcgH\cong\SO(2)\times\Spin(n)$ is the stabilizer of $\cx_0$ from Lemma \ref{lemma groups in covering}. The kernel of $\tcgK\xrightarrow{\rho_{2,n+2}|_{\tcgK}}\cgK\ra\gK$, where  the second map is the canonical projection, is the center of $\tcgK$. So (up to isomorphism) this is the homomorphism $\rho_{n+2}^c$ from (\ref{first ses}) and this induces a $\tcgK$-action on $\rgr$. Let $\gH^c$ be the stabilizer of  $x_0=[e_1,e_2]^+$ inside $\tcgK$.

\begin{lemma}\label{lemma M as homogeneous space}
The group $\gH^c$ is isomorphic to $\SO(2)\times\Spin^c(n)$ and   $\tcgH\subset\gH^c$  corresponds to the standard inclusion $\SO(2)\times\Spin(n)\hookrightarrow\SO(2)\times\Spin^c(n)$. 
\end{lemma}
\begin{proof}
By construction, it  is clear that $\tcgH$ is a subgroup of $\gH^c$ and by definition, $\gH^c$ is the preimage of the stabilizer $\gH$ of $x_0$ in $\gK$. The restriction of $\tcgK\ra\cgK$ to $\gH^c$ induces a short exact sequence $0\ra\Z_2\ra\gH^c\ra\SO(2)\times\SO(2)\times\SO(n)\ra0$ where the last group is the usual block diagonal subgroup of $\cgK$. It is easy to  see that $\gH^c$ is isomorphic to the quotient of $\gU(1)\times\gU(1)\times\Spin(n)$ by the subgroup generated by $(-1,-1,1)$ and $(-1,1,-1)$. Let $\langle e^{it},e^{is},a\rangle$ be the class of $(e^{it},e^{is},a)$ in the quotient so that the map $\rho_{n+2}^c|_{\gH^c}:\gH^c\ra\gH$ is  $\langle e^{it},e^{is},a\rangle\mapsto(\rho_2(e^{is}),\rho_n(a))$.
It is straightforward to verify that $\gH^c\ra\SO(2)\times\Spin^c(n), \ \langle e^{it},e^{is},a\rangle\mapsto (\rho_2(e^{is}),\langle e^{i(s-t)},a\rangle)$ is a bijective homomorphism of Lie groups where we use the notation set in Section \ref{section spin group}. The inverse map is $(e^{iu},\langle e^{iv},a\rangle)\mapsto\langle e^{i(\frac{u}{2}-v)},e^{\frac{iu}{2}} ,a\rangle$. 
Comparing this with (\ref{stabilizer of cx0}), we have that $\langle e^{it},e^{is},a\rangle\in\tcgH$ if and only if $\rho_2(e^{it})=\rho_2(e^{is})$. Hence, $\tcgH$ corresponds to the subgroup $\SO(2)\times\Spin(n)$ as claimed. \end{proof}

\medskip

We see that $\rgr\cong\tcgK/\gH^c$ and that the map $q:\sv\ra\rgr$ intertwines the actions of $\tcgK$ on both spaces. As $\gH^c$ is a subgroup of $\G_0^c$, we can extend the associated homogeneous principal bundle $\gH^c\ra\tcgK\ra\rgr$  to a $\G_0^c$-principal bundle $\cG_0^c:=\tcgK\times_{\gH^c}\G_0^c\xrightarrow{p_0^c}\rgr$. This is again a homogeneous principal bundle, i.e. there is a  canonical left $\tcgK$-action that covers the action on $\rgr$. 

\bigskip

\begin{thm}\label{thm descended complex}
Suppose that $n\ge3$.

(i) The homogeneous $\G_0^c$-principal bundle $p_0^c:\cG_0^c\ra\rgr$  constructed above is a lift of the principal bundle of the homogeneous PCS-structure of Grassmannian type from Theorem \ref{thm pcs quotient sc to rgr}.
\smallskip

(ii) The associated  determinant circle bundle is isomorphic to $q:\sv\ra\rgr$ and $\tcp_0:\cG_0^c\ra\sv$ 
is isomorphic to the principal bundle of the  homogeneous Lie contact structure  of type $(\tilde\G,\tilde\gP)$ on $\sv$.

\smallskip
(iii) Let $(\mV_{\lambda_i},\varrho^c_i)$ be the $\G_0^c$-modules that are via (\ref{spinc representation on spinors}) associated to the $\tilde\G_0$-modules $(\mV_{\lambda_i},\varrho_i)$ from Theorem \ref{thm 2-Dirac complex}. Put  $V_i:=\cG_0^c\times_{\G^c_0}\mV_{\lambda_i}$. Then the 2-Dirac complex on $\sv$ descends to a complex 
\begin{equation}\label{descended complex}
\Gamma(V_1)\ra\Gamma(V_2)\ra\Gamma(V_3)\ra\Gamma(V_4)
\end{equation}
of $\tcgK$-invariant linear  differential operators. The first operator in the sequence is the 2-Dirac operator  associated to the $\G_0^c$-structure.
\end{thm}
\begin{proof}
(i) By the proof of Lemma \ref{lemma M as homogeneous space}, it follows that $\rho_{n+2}^c|_{\gH^c}:\gH^c\ra\gH$ equals to $\Id_{\SO(2)}\times\rho_n^c$. It is then  straightforward to verify that the canonical map $\tcgK\times\G_0^c\ra\gK\times_\gH\G_0\cong\cG_0$ descends to a bundle map $\cG_0^c\ra\cG_0$ which is compatible with $\G_0^c\ra\G_0$.

(ii)  The total space of the associated circle bundle is $\cG_0^c\times_{\G_0^c}\gU(1)\cong \tcgK\times_{\gH^c}\G_0^c\times_{\G_0^c}\gU(1)\cong\tcgK\times_{\gH^c}\gU(1)$. The short exact sequence $0\ra\tcgH\ra\gH^c\ra\gU(1)\ra0$ shows that the last space is the quotient of $\tcgK$ by the right action of $\tcgH$, i.e.  this is $\tcgK/\tcgH\cong\sv$. Similarly we find that $\cG^c_0=\tcgK\times_{\gH^c}\G_0^c\cong\tcgK\times_{\SO(2)}\GL^+(2,\R)\cong\tcgK\times_{\tcgH}\tilde\G_0$ and the last space is isomorphic to the total space of the principal bundle of the homogeneous  Lie contact structure of type $(\tilde\G,\tilde\gP)$.

(iii) Let $\alpha$ be the contact form associated to the transversal infinitesimal symmetry $\xi$ from Proposition \ref{thm pcs quotient sc to rgr}. Then on one hand, $q:\sv\ra\rgr$ is a principal $\gU(1)$-bundle with a principal connection $i\alpha$ and on the other hand, it is the map underlying the PCS-quotient from Proposition \ref{thm pcs quotient sc to rgr}. Hence, the underlying $cs$-structure on $\rgr$ is trivialized by a global symplectic form $\mathbf{\Omega}$ such that $d\alpha=q^\ast\mathbf{\Omega}$ and $\frac{-1}{2\pi}[\mathbf{\Omega}]=c_1(\sv)$. By the part (ii), Proposition \ref{thm compatible G0c structure} and Theorem  \ref{thm descending}, the 2-Dirac complex descends to a sequence of operators on $\rgr$.  From the proof of Theorem \ref{thm descending} follows that (\ref{descended complex}) is  a complex. As the operators in the 2-Dirac complex are $\tcgK$-invariant, also the operators in the descended complex have this symmetry.
\end{proof}

\subsection{Symbol sequence}\label{section symbol sequence}
We know that  $T^\ast\rgr\cong E\otimes F^\ast$ where $E$ and $F$ are the tautological vector bundles of rank 2 and $n$, respectively.  By Remark \ref{remark bundles in Dirac complex},  as vector spaces: $\mV_0\cong\mV_3\cong\Sp$ and $\mV_1\cong\mV_2\cong\R^2\otimes\Sp$. A choice of $u\in\cG_0^c$ in the fiber over $x\in\rgr$  induces isomorphisms $(V_0)_x\cong(V_1)_x\cong\Sp, (V_1)_x\cong(V_2)_x\cong\R^2\otimes\Sp$ and a  metric on $F_x$ which is compatible with $\ell$. Hence,  $T^\ast_x\rgr\cong \R^2\otimes \R^n$ and we can view  $X\in T^\ast_x\rgr$  as a pair of vectors  $(X_1,X_2)$  from $\R^n$. 

\begin{lemma}\label{lemma symbol sequence}
With the notation introduced above, the sequence of symbols of the complex (\ref{descended complex}) associated  to $X\in T^\ast_x\rgr$ is
 \begin{equation}\label{symbol sequence}
  \Sp\xrightarrow{\sigma(D_1,X)}\R^2\otimes\Sp\xrightarrow{\sigma(D_2,X)}\R^2\otimes\Sp\xrightarrow{\sigma(D_3,X)}\Sp
 \end{equation}
where 
\begin{align}
\sigma(D_1,X)(\psi)&=(X_1.\psi,X_2.\psi),\\
\sigma(D_2,X)(\psi_1,\psi_2)&=(-X_2.X_1.\psi_1+X_1.X_1.\psi_2,-X_2.X_2.\psi_1+X_1.X_2.\psi_2),\nonumber\\
\sigma(D_3,X)(\phi_1,\phi_2)&=-X_2.\phi_1+X_1.\phi_2\nonumber
\end{align}
where the dot denotes  the standard action of $\R^n$ on $\Sp$ and we view $\R^2\otimes\Sp$ as the vector space of pairs of spinors.
\end{lemma}
\begin{proof}
The point $u\in\cG_0^c$ also  determines isomorphisms $(V^\sharp_i)_{\cx}\cong(V_i)_x,\ i=0,1,2,3$ where $\cx=\tcp_0(u)$ and $q(\cx)=x=p_0^c(u)$. Since the map $Tq$ descends to a linear isomorphism $H_{\cx}\ra T_{x}\rgr$ and dually to $T^\ast_x\rgr\cong H^\ast_{\cx}$, we  can view $X$ as a vector in $H^\ast_{\cx}$. The claim then follows from \cite[Section 6.3]{S}. 
\end{proof}

Notice that the 2-Dirac complex is in \cite{S}  computed over an affine  subset of the homogeneous space with respect to the very flat Weyl structure and this is not the Weyl structure determined by the transversal infinitesimal symmetry $\xi$. However, the symbol of each operator in the complex is independent of the choice of  Weyl structure and so we can use in the proof of Lemma \ref{lemma symbol sequence} the result.

\begin{thm}\label{thm ellipticity and index}
The complex (\ref{descended complex}) is elliptic and its index  is equal to zero.
\end{thm}
\begin{proof}
It can can be verified directly that the symbol sequence is exact. However, it is more convenient to argue as in the proof of  Lemma \ref{lemma symbol sequence} since then it is clear that the exactness of (\ref{symbol sequence}) is equivalent to the exactness of the symbol sequence of the 2-Dirac complex  with respect to each non-zero vector from $H^\ast$. The latter claim is proved in  \cite[Section 6.3]{S}.

It remains to show that the index is equal to zero. We reduce the structure group of $\cG_0^c\ra\rgr$ to $\gH^c$.  Then $\mV_0\cong\mV_3$ and $\mV_1\cong\mV_2$ are isomorphic as $\gH^c$-modules. Then we can compute (see \cite{AS}) the index using the Atiyah-Singer index formula for $G$-structures. Here we need that $\rgr$ is an orientable compact manifold and that $T\rgr$ is associated to $\gH^c$-structure $\tcgK\ra\rgr$. Using the duality in the complex, the claim follows.
\end{proof}

\def\bibname{Bibliography}

\bigskip

\textsc{Mathematical Institute of  Charles University, Sokolovsk\'a  49/83, Prague, The Czech Republic.}\\
\smallskip

\textit{E-mail address}: \texttt{salac$@$karlin.mff.cuni.cz}
\end{document}